\documentclass[11pt,twoside]{article} 	
\usepackage{geometry}  		
\usepackage{graphicx,subcaption}					
\usepackage{amssymb}
\usepackage{indentfirst}
\usepackage{amsmath}
\usepackage{amsthm}
\usepackage{bm}
\usepackage{changepage}
\usepackage{lineno}
\usepackage{setspace}
\usepackage{booktabs,multirow}
\usepackage{authblk}
\usepackage{graphicx}
\usepackage{float}
\usepackage[flushleft]{threeparttable} 
\usepackage{tikz-cd}
\usepackage{cite}
\usepackage{url}
\usepackage{lipsum}
\usepackage[marginal]{footmisc}
\usepackage{fancyhdr}

\usepackage[colorlinks = true,
            linkcolor = blue,
            urlcolor  = blue,
            citecolor = blue,
            anchorcolor = blue]{hyperref}
\renewenvironment{abstract}{
    \setlength\textwidth{3in}
    \begin{quote}
}
{ \end{quote}}

\numberwithin{equation}{section}
\newtheorem{theorem}{Theorem}[section]
\newtheorem{proposition}[theorem]{Proposition}
\newtheorem{definition}[theorem]{Definition}
\newtheorem{lemma}[theorem]{Lemma}
\newtheorem{corollary}[theorem]{Corollary}
\newtheorem{claim}{Claim}
\newtheorem{remark}[theorem]{Remark}

\newcommand\keywords[1]{\textbf{Keywords:}#1}

\title{On the fundamental group of steady gradient Ricci solitons with nonnegative sectional curvature}
\author{Yuxing $\text{Deng}^*$ and~~Yuehan~Hao }

\date{}							

\begin{document}

	\maketitle
    \pagestyle{fancy}
    \fancyhf{}
    \fancyhead[OC]{\small FUNDAMENTAL GROUP OF STEADY GRADIENT RICCI SOLITONS}
    \fancyhead[EC]{\small YUXING DEND~AND~YUEHAN~HAO}
    \renewcommand{\headrulewidth}{0pt}
    \fancyfoot[c]{\small \thepage}
    \footnote{
    $^{*}$Supported by National Key R$\&$D Program of China 2022YFA1007600.

    }%

	\begin{abstract}
        ABSTRACT. In this paper, we study the fundamental group of the complete steady gradient Ricci soliton with nonnegative sectional curvature. We prove that the fundamental group of such a Ricci soliton is either trivial or infinite. As a corollary, we show that an $n$-dimensional complete $\kappa$-noncollapsed steady gradient Ricci soliton with nonnegative sectional curvature must be diffeomorphic to $\mathbb{R}^n$. 
    
    \end{abstract}

	

	\section{Introduction} 
	

	 A \textbf{Ricci soliton} is a quadruple $(M ,g,X,\lambda)$ consisting of a smooth manifold $M $, a Riemannian metric $g$, a smooth vector field $X$ and a real constant $\lambda$, if these variables satisfy the equation
	\begin{equation}
		\label{ricci soliton definition}
		{\rm Ric}+\mathcal{L}_Xg=\frac{\lambda}{2}g
	\end{equation}
	on $M$, where ${\rm Ric}$ denotes the Ricci tensor of $g$, and $\mathcal{L}$ denotes the Lie derivative. 
	For $\lambda =0$ the Ricci soliton is \textbf{steady}, for $\lambda >0$ it is \textbf{shrinking} and for $\lambda <0$ \textbf{expanding}.
	Up to scaling, we can normalize $\lambda = 0,1,$ or $-1$, respectively.

	If the vector field $X= \nabla f$ for some smooth function $f$ on $M $, then the Ricci soliton is called a \textbf{gradient Ricci soliton}.
	For such a soliton, equation (\ref{ricci soliton definition}) simplifies to 
	\begin{equation}
		\label{ric+downtri f=}
		{\rm Ric} +\nabla^2f=\frac{\lambda}{2} g,
	\end{equation}
	since $\mathcal{L}_{\nabla f} g = 2 \nabla^2 f$. Here $\nabla^2$ denotes the Hessian of $f$.


    Since a Ricci soliton that arises as the blow-up limit of a compact Ricci flow is $\kappa$-noncollapsed. We are interested in  $\kappa$-noncollapsed Ricci solitons. In dimension 3, the uniqueness of $\kappa$-noncollapsed and non-flat steady gradient Ricci solitons was claimed by Perelman in \cite{perelman_entropy_2002}. Perelman's claim was confirmed by Brendle in \cite{brendle_rotational_2013}. There are also many works on the higher dimensional  $\kappa$-noncollapsed steady gradient Ricci solitons, such as \cite{Deng-Zhu-2018,Cao_chen_2012,brendle_rotational_2014,Deng-Zhu-Mathann-2020,Deng-Zhu-SCM-2020,Deng-Zhu-JEMS-2020,Bamler-Chow-Deng-Ma-2022,Chow-Deng-Ma-2022,chan_dimension_2023,zhao-zhu-arxiv-2022,zhao-zhu-arxiv-2024,ma-Mahmoudian-Sesum-arxiv-2023}. Many new examples of steady gradient Ricci solitons have also been found, see \cite{Lai-JDG-2024,appleton_family_2022,conlon-Deruelle-arxiv-2020,buttsworth-2021-arxiv,Buzano-Dancer-Wang-2015,wink-2023} and references therein. Due to the steady Ricci solitons constructed by Yi Lai \cite{Lai-JDG-2024}, 4-dimensional $\kappa$-noncollapsed steady gradient Ricci soliton with nonnegative sectional curvature and positive Ricci curvature is not unique. How to classify these Ricci solitons is still an open problem. Even the topology of 4-dimensional $\kappa$-noncollapsed steady gradient Ricci soliton with nonnegative sectional curvature has not been classified yet.

    It is well-known that the topology of manifolds with nonnegative sectional curvature has been studied by Cheeger and Gromoll in \cite{Cheeger-Gromoll-1972}. A manifold with nonnegative sectional curvature must be a normal bundle over a soul\cite{Cheeger-Gromoll-1972,sarafutdinov-1979}. For the expanding case, an expanding gradient Ricci soliton with nonnegative Ricci curvature must be diffeomorphic to the Euclidean space, see Lemma 5.5 in \cite{Cao-Catino-Chen-Mantegazza-Mazzieri-2014} or Lemma 2.2 in \cite{Chen-Deruelle-2015}. For the shrinking case, the fundamental group of a shrinking gradient Ricci soliton is finite \cite{Wylie_2008}.  There are also many noncompact shrinking gradient Ricci solitons with nonnegative sectional curvature and nontrivial finite fundamental groups, such as $(\mathbb{S}^2\times \mathbb{R})/\mathbb{Z}_2$.  In this paper, we study the topology of steady gradient Ricci solitons with nonnegative sectional curvature. We get the following result on the fundamental groups of steady gradient Ricci solitons.

    \begin{theorem}
  	    \label{MAIN}
		Suppose that $(M,g,f)$ is an $n$-dimensional complete steady gradient Ricci soliton with nonnegative sectional curvature. Then the fundamental group $\pi_1(M)$ of $M$ is either trivial or infinite. 

        Moreover, if $\pi_1(M)$ is trivial, then $M$ is diffeomorphic to $\mathbb{R}^n$.
	\end{theorem}

    Bamler\cite{bamler_fundamental_2021} has proved that any $\kappa$-noncollapsed ancient soliton to the Ricci flow must have a finite fundamental group.
	\begin{theorem}[Theorem 1.1. \cite{bamler_fundamental_2021}]
		\label{bamler_noncollapsed finite fundamental group_theorem}
		Let $(M,(g_t)_{t \le 0})$ be a $\kappa$-noncollapsed ancient Ricci flow that has complete time-slices and bounded curvature on compact time-intervals. Then the fundamental group $\pi_1(M)$ is finite.
	\end{theorem}

    As an application of Bamler's finiteness theorem and Theorem \ref{MAIN}, We have

    \begin{theorem}\label{MAIN-2}
		Any $n$-dimensional $\kappa$-noncollapsed complete steady gradient Ricci soliton with nonnegative sectional curvature is diffeomorphic to $\mathbb{R}^n$.
	\end{theorem}

    By Theorem \ref{MAIN-2}, we have the following refinement of Corollary 1.2 by Cao and Xie in \cite{cao-xie-2024-arxiv}.
    \begin{corollary}
        Let $(M^4,g,f)$ be a $4$-dimensional complete noncompact, non-flat, $\kappa$-noncollapsed steady gradient Ricci soliton with nonnegative isotropic curvature. Then, either
        
			(i) $(M^4,g,f)$ has positive isotropic curvature and positive Ricci curvature, or 
            
			(ii) $(M^4,g,f)$ is isometric to the product $N^3\times\mathbb{R}$ of the 3D Bryant soliton with the real line. 
		
    \end{corollary}
	
	 
	We now outline the main steps involved in the proof of the main result. In Section 2, we briefly review the geometry of the universal Riemannian covering of manifolds. In Section 3, we prove that the isometry groups of universal covers based on gradient Ricci solitons with nonnegative sectional curvature preserve the splitting property argued in \cite{guan_rigidity_nodate}. In addition to positive Ricci curvature, we argue that such solitons are simply connected and do not have non-trivial finite quotient manifolds in Section 4.  In Section 5, we present the proof of Theorem \ref{MAIN}. Suppose that $M$ is a complete steady gradient Ricci soliton with nonnegative sectional curvature. Applying the splitting theorem in \cite{guan_rigidity_nodate} and analyzing the behaviour of covering automorphisms acting on factors of its universal cover separately, we obtain the main theorem.

	\section{Preliminaries}
	
	In this section, we recall some results concerning covering automorphism groups of Riemannian coverings. Given a Riemannian covering map $\pi: (\tilde{M}, \tilde{g}) \rightarrow (M,g)$. The covering automorphism group is defined as follows.

    \begin{definition}[Covering automorphism groups of Riemannian covering maps]
		\label{Covering Automorphism Groups}
		Let $(\tilde{M}, \tilde{g})$ and $(M, g )$ be Riemannian manifolds and $\pi: \tilde{M} \rightarrow M$ be a Riemannian covering map. A \textbf{covering automorphism} (or \textbf{deck transformation}) of $\pi$ is an isomorphism $\varphi: \tilde{M} \rightarrow \tilde{M}$ such that $\pi \circ \varphi = \pi$; 
		the set of all covering automorphisms is a group acting on $\tilde{M}$ on the left, under composition of isomorphisms, 
		called the \textbf{covering automorphism group} (or \textbf{deck transformation group}) and denoted by ${\rm Aut}_{\pi}(\tilde{M})$.
	\end{definition}

    
    Obviously, the covering automorphism group ${\rm Aut}_{\pi}(\tilde{M})$ is a subgroup of the isometry group ${\rm Iso}(\tilde{M},\tilde{g})$. ${\rm Aut}_{\pi}(\tilde{M})$ acts smoothly, freely and properly on $\tilde{M}$.

	\begin{proposition}[\cite{lee_introduction_2012,petersen_riemannian_2006}]
		\label{Aut:smooth+free+proper+isometrical proposition}
		Let $\pi: (\tilde{M}, \tilde{g}) \rightarrow (M , g )$ be a Riemannian covering map over a Riemannian manifold $M$.  ${\rm Aut}_{\pi}(\tilde{M})$ is a discrete Lie group acting smoothly, freely, properly and isometrically on $\tilde{M}$.
	\end{proposition}

    ${\rm Aut}_{\pi}(\tilde{M})$ is also closely related to the fundamental groups of  $M$. 
	
    \begin{theorem}[Corollary 12.9. \cite{lee_introduction_2011}]
        \label{theorem-structure theorem on universal covering map}
		Suppose that $\pi : \tilde{M} \rightarrow M$ is a universal covering map, 
		$\tilde{x} \in \tilde{M}$, 
		and $x= \pi(\tilde{x})$. 
		 Then, we have:
		\begin{equation}
			{\rm Aut}_{\pi}(\tilde{M}) \cong \pi_1(M,x).\notag
		\end{equation} 
		 
	\end{theorem}

	As a result of Proposition \ref{Aut:smooth+free+proper+isometrical proposition} and Theorem \ref{theorem-structure theorem on universal covering map}, the following holds.

	\begin{proposition}[Corollary 2.33. \cite{lee_introduction_2018}]\label{propositon-quotients isometric to universal cover module deck group}
		Suppose $(M,g)$ and $(\tilde{M}, \tilde{g})$ are Riemannian manifolds, and $\pi : \tilde{M} \rightarrow M$ is a universal Riemannian covering map. Then $M$ is isometric to $\tilde{M}/{\rm Aut}_{\pi}(\tilde{M})$.
	\end{proposition}

	\section{The isometry group of the universal cover}

 Given a steady gradient Ricci soliton $(M,g,f)$, suppose $(\tilde{M},\tilde{g})$ is the universal cover of $(M,g)$. By  Theorem \ref{theorem-structure theorem on universal covering map},
    \begin{align}
         \pi_1(M,x) \cong {\rm Aut}_{\pi}(\tilde{M})\subseteq {\rm Iso}(\tilde{M},\tilde{g})\notag
    \end{align}
    So, to figure out the structure of $\pi_1(M,x)$, we need to study the isometry group ${\rm Iso}(\tilde{M},\tilde{g})$.
	In \cite{guan_rigidity_nodate}, Guan-Lu-Xu prove a splitting theorem for gradient Ricci solitons with nonnegative sectional curvature.
	\begin{theorem}[Theorem 1.1. \cite{guan_rigidity_nodate}]
        \label{split theorem_universal}
		Let $(M^n ,g ,f)$ be a gradient Ricci soliton. If $g$ has nonnegative sectional curvature, then the rank of the Ricci curvature is constant. Thus, either the  Ricci curvature is strictly positive or the universal covering $(\tilde{M}, \tilde{g}) = (N^k,h) \times \mathbb{R}^{n-k}$ splits isometrically and $(N^k, h)$ has strictly positive Ricci curvature.
	\end{theorem}

    Inspired by this, we want to study the relations of ${\rm Iso}(N,g)$, ${\rm Iso}(\mathbb{R}^k)$ and ${\rm Iso}(\tilde{M}, \tilde{g})$. In general,  ${\rm Iso}(M,g)$ is not the product of  ${\rm Iso}(M_1,g_1)$ and ${\rm Iso}(M_2,g_2)$ when $M=(M_1,g_1)\times(M_2,g_2)$. For example, given any irreducible Riemannian manifold $(M,g)$, ${\rm Iso}(M,g)\times {\rm Iso}(M,g)$ is a proper subgroup of ${\rm Iso}(M\times M,g\oplus g)$. However, also in view of Theorem \ref{split theorem_universal}, we can obtain the following decomposition result due to the difference between the curvatures on $N^k$ and $\mathbb{R}^{n-k}$.
	
	\begin{proposition}
    		\label{proposition-isometry separately N R^n}
		Suppose that $N$ is a simply connected complete Riemannian manifold with positive Ricci curvature. Then, we have ${\rm Iso} (N \times \mathbb{R}^n) ={\rm Iso} ( N ) \times {\rm Iso} ( \mathbb{R}^n)$ for $n\in \mathbb{N}$.
	\end{proposition}
	\begin{proof}
		Let $\varphi: N  \times \mathbb{R}^n \rightarrow N  \times \mathbb{R}^n $ be 
		an isometry of  $N  \times \mathbb{R}^n $. 
		Fix a point $ (x, y) \in N  \times \mathbb{R}^n $  and 
		consider the differential of $\varphi$: 
		\begin{equation}
			d \varphi_{(x, y)}: 
			T_{(x, y)}(N  \times \mathbb{R}^n ) 
			\cong 
			T_{x} N  \oplus T_{y} \mathbb{R}^n  
			\rightarrow 
			T_{\left(x^{\prime}, y^{\prime}\right)}(N  \times \mathbb{R}^n ) 
			\cong
			T_{x^{\prime}} N  \oplus T_{y^{\prime}} \mathbb{R}^n,
		\end{equation}
		where  $\left(x^{\prime}, y^{\prime}\right)=\varphi(x, y) $.
		
		\begin{claim}
			$d\varphi$ sends  $T_{x}  N  $ to  $T_{x'}  N  $ and   $T_{y} \mathbb{R}^n $ to $T_{y'} \mathbb{R}^n $.
			\label{isometrysep}
		\end{claim}
		
		\begin{proof}[Proof of claim \ref{isometrysep}]
			Recall from Proposition 2.5 of Chapter IV in \cite{kobayashi_foundations_1996} that if  $\varphi$ is an isometry of $N  \times \mathbb{R}^n$, 
			then the differential of $\varphi$ commutes with parallel transports.
			More precisely, if $\gamma $ is a loop based at  $(x,y)$ in $N  \times \mathbb{R}^n$, then the following diagram is commutative:
			\begin{equation}
				\begin{tikzcd}
					T_{(x,y)}(N \times \mathbb{R}^n) \arrow[d,"d\varphi"] \arrow[r,"\gamma"] & T_{(x,y)}(N  \times \mathbb{R}^n)  \arrow[d,"d\varphi"]\\
					T_{(x',y')}(N  \times \mathbb{R}^n) \arrow[r,"\gamma'"] & T_{(x',y')}(N  \times \mathbb{R}^n),   
				\end{tikzcd}
			\end{equation}
			where  $\gamma' = \varphi(\gamma)$. Regarding this, $d\varphi$ sends a canonical decomposition of $T_{(x,y)}( N  \times \mathbb{R}^n)$ to a canonical decomposition of $T_{\varphi(x,y)}( N \times \mathbb{R}^n)$. 
			
            Note that sectional curvatures are invariant under isometries. Since $N$ has positive Ricci curvature, $d\varphi$ sends  $T_{x}  N  $ to  $T_{x'}  N  $ and   $T_{y} \mathbb{R}^n $ to $T_{y'} \mathbb{R}^n $ separately. Here, we use the de Rham decomposition.
		\end{proof}
		As an isometry maps geodesics to geodesics,  $\varphi$  sends totally geodesic submanifolds $N  \times\{y\} $ to $ N  \times\left\{y^{\prime}\right\} $ and $\{x\} \times \mathbb{R}^n$  to  $\left\{x^{\prime}\right\} \times \mathbb{R}^n $. Suppose $\pi_{N}$ and $\pi_{\mathbb{R}^n}$ are the projection from $N \times \mathbb{R}^n$ onto $N$ and $\mathbb{R}^n$, respectively. 
		We define $\varphi_1 : N \rightarrow N$ by $\varphi_1(u)= \pi_N \circ \varphi(u,y)$  and $\varphi_2 : \mathbb{R}^n \rightarrow \mathbb{R}^n$ by $\varphi_2(v)= \pi_{\mathbb{R}^n} \circ \varphi(x,v)$, where $u \in N$ and $v \in \mathbb{R}^n$. 

        Considering that  $\varphi_{1} \in {\rm Iso} (N)$  and  $\varphi_{2} \in {\rm Iso} (\mathbb{R}^n) $, 
            $ \left(\varphi_{1}, \varphi_{2}\right)$ is a product isometry of  $N \times \mathbb{R}^n $. 
		Isometries $\varphi$ and $\left(\varphi_{1}, \varphi_{2}\right)$ satisfy $\varphi(x,y) = \left(\varphi_{1}, \varphi_{2}\right)(x,y)$ and $d\varphi_{(x,y)} = d\left(\varphi_{1}, \varphi_{2}\right)_{(x,y)}$. Hence they coincide globally and we complete the proof.
	\end{proof}

	\section{Steady gradient Ricci solitons with nonnegative sectional curvature and positive Ricci curvature} 
    In this section, we will study steady gradient Ricci solitons with nonnegative sectional curvature and positive Ricci curvature. If the scalar curvature attains its maximum and the Ricci curvature is positive, then $M^n$ is diffeomorphic to $\mathbb{R}^n$\cite{Cao_chen_2012}. On a steady gradient Ricci soliton, the scalar curvature attains its maximum if and only if 
 the potential $f$ has a critical point $x_0$, i.e., $|\nabla f|(x_0)=0$. It is still unknown whether there exists a critical point on $M$, even if we assume the sectional curvature is nonnegative. Given a steady gradient Ricci soliton $(M^n,g,f)$, denote the level set of a function $f$ on $M$, by $M_s :=\{ x\in M: f(x) =s\}$. When $M_s$ is compact, we show the existence of the critical point in subsection \ref{subsection-compact}. When $M_s$ is noncompact, we will show that the level set is diffeomorphic to $\mathbb{R}^{n-1}$ in subsection \ref{subsection-noncompact}. We will also show that a steady gradient Ricci soliton with nonnegative sectional curvature and positive Ricci curvature has only trivial quotient in subsection \ref{subsection-quotient}.
    
	\subsection{Compact level sets of \textit{f}}\label{subsection-compact}

	Counting the ends of a noncompact manifold is an important problem in the study of noncompact manifolds. For the definition of ends, one may refer to\cite{chow_ricci_2023}. In \cite{gromoll_complete_1969}, Gromoll and Meyer have proved that a complete open manifold $M$ of positive Ricci curvature has only one end. Munteanu and Wang generalize Gromoll and Meyer's results on steady gradient Ricci solitons. They prove that a steady gradient Ricci soliton is either connected at infinity or isometric to $N\times \mathbb{R}$, where $N$ is a compact Ricci flat manifold (See Theorem 4.2 in \cite{Munteanu_wang_2011}). This result gives a strong restriction of the topology of steady gradient Ricci solitons.
	\begin{lemma}\label{lemma-existence of critical point}
		Suppose $(M,g,f)$ is an $n$-dimensional complete steady gradient Ricci soliton.  Assume that $f$ has a compact level set $M_s$ for some $s \in f(M)$. 
		Then, there exists a point $x \in M$, such that $|\nabla f|(x) = 0$. 
	\end{lemma}
	\begin{proof}
We prove this by contradiction. Suppose $|\nabla f|(x)\neq0$, $\forall~x\in M$. Then, $f$ must be an open map. Hence, $f(M)$ is a connected and open subset of $\mathbb{R}$. We may assume that $f(M)=(a,b)$, where $-\infty\le a<b\le +\infty$. Note that $s\in (a,b)$ and $|\nabla f|$ is not vanishing everywhere. Then, $M$ is diffeomorphic to $(a,b)\times M_s$. Hence, $M$ has two ends when  $M_s$ is compact. It contradicts Theorem 4.2 in \cite{Munteanu_wang_2011}.
	\end{proof}

By Lemma \ref{lemma-existence of critical point}, we can show that $M$ is diffeomorphic to the Euclidean space when the level set is compact.
\begin{theorem}\label{theorem-level set compact}
		Suppose $(M,g,f)$ is an $n$-dimensional complete steady gradient Ricci soliton with positive Ricci curvature. 
		If  $M_s$ is compact for some $s \in f(M)$,
		then, $M$ is diffeomorphic to $\mathbb{R}^n$.
	\end{theorem}
  
\begin{proof}
Hamilton\cite{hamilton_formation_1995} has proved that a steady gradient Ricci soliton satisfies 
	\begin{equation}\label{identity-hamilton}
		R+|\nabla f|^2 = C,
	\end{equation}
	where $C$ is a constant. Here $R$ denotes the scalar curvature. 
   By lemma \ref{lemma-existence of critical point}, $f$ has a critical point $x_0$. So, the scalar curvature attains its maximum at point $x_0$ by (\ref{identity-hamilton}). By Proposition 2.3 in \cite{Cao_chen_2012}, $M$ is diffeomorphic to $\mathbb{R}^n$.
\end{proof}

	\subsection{Noncompact level sets of \textit{f}}\label{subsection-noncompact}
	According to the above, if an $n$-dimensional complete steady gradient Ricci soliton $(M,g,f)$ with positive Ricci curvature has no critical point, then all level sets of $f$ are noncompact.
	
	\begin{theorem}\label{theorem-noncompact level set}
		Suppose $(M,g,f)$ is an $n$-dimensional complete steady gradient Ricci soliton with nonnegative sectional curvature and positive Ricci curvature. If  $M_s$ is noncompact for some $s \in f(M)$,
		then $M$ is diffeomorphic to $\mathbb{R}^n$.
	\end{theorem}
	\begin{proof}
If there is a point $x_0 \in M$ such that $|\nabla f|(x_0)=0$, then we can show that $M$ is diffeomorphic to $\mathbb{R}^n$ by the same argument as in the proof of Theorem \ref{theorem-level set compact}.  

Now, we assume that $|\nabla f|$ is not vanishing everywhere. Note that $M_s$ is a smooth manifold due to the regular value theorem for any $s \in f(M)$. Then, we follow the computation in the Lemma 2.1 of \cite{chan_dimension_2023}. The second fundamental form ${\rm II}$ of $M_s$ is determined by 
		\begin{equation}
			{\rm II}(X, Y ) := <\nabla_X {\rm n},Y> = \frac{1 }{|\nabla f |}{\rm Ric}(X, Y )  
			\label{II>0}
		\end{equation} 
		for any $X, Y$ tangent to $M_s$, where ${\rm n}= -\frac{\nabla f}{|\nabla f|}$. Then ${\rm II}>0$.
		Moreover, 
		\begin{equation}
			{\rm sec}_{M_s} (X, Y ) = {\rm sec}_{M}(X, Y ) + {\rm II}(X, X){\rm II}(Y, Y )-({\rm II}(X, Y ))^2 > 0,  
		\end{equation}
		where ${\rm sec}$ denotes sectional curvature, and $X,Y$ are unit vectors tangent to $M_s$ with $X \perp Y$.
		More precisely, $M_s$ is an $(n-1)$-manifold with positive sectional curvature. Since Lemma \ref{lemma-existence of critical point} implies that all level sets of $f$ are noncompact, $M_s$ is diffeomorphic to an Euclidean space. Hence,  we obtain that $M$ is diffeomorphic to $\mathbb{R}^n$ followed by the same argument as in the proof of Lemma \ref{lemma-existence of critical point}. 
	\end{proof}
	 As a corollary of Theorem \ref{theorem-level set compact} and Theorem \ref{theorem-noncompact level set}, we obtain the following main result of this section.
	 \begin{corollary}\label{corollary-Ricci positive case is simply connected}
	 	Suppose $(M,g,f)$ is an $n$-dimensional complete steady gradient Ricci soliton with nonnegative sectional curvature and positive Ricci curvature. Then, $M$ is diffeomorphic to $\mathbb{R}^n$.
	 	
	 \end{corollary}

     \begin{remark}
        For $n=4$, Theorem \ref{theorem-noncompact level set} and Corollary \ref{corollary-Ricci positive case is simply connected} still hold if we only assume $(M,g,f)$ has positive Ricci curvature. Note that any $3$-dimensional noncompact manifold with positive Ricci curvature is diffeomorphic to $\mathbb{R}^3$(cf.\cite{Schoen-Yau-1982,Liu-2013} ). So, the argument in Theorem \ref{theorem-noncompact level set} and Corollary \ref{corollary-Ricci positive case is simply connected} remains true.
     \end{remark}

     \subsection{Quotients of steady gradient Ricci solitons}\label{subsection-quotient}

     To prove Theorem \ref{MAIN}, we still need to study the quotients of steady gradient Ricci soltions with nonnegative sectional curvature and positive Ricci curvature.

     \begin{theorem}\label{theorem-symmetry of f}
Let $(M,g,f)$ be a complete gradient Ricci soliton. Suppose $G$ is a finite subgroup of ${\rm Iso}(M,g)$. Then, one of the following holds:

 (1) $(M,g)$ locally splits off a line;
 
 (2) $f(x)=f(\phi(x))$ for all $x\in M$ and $\phi\in G$.
     \end{theorem}

     \begin{proof}
Let $\phi\in G$ and $V_1,V_2\in T_xM$ for some $x\in M$. Assume $y=\phi(x)$. By the Ricci soliton equation and $\phi\in {\rm Iso}(M,g)$, we have 
\begin{align}
-\nabla_{V_1}\nabla_{V_2}f|_x=&{\rm Ric}(V_1,V_2)|_x-\frac{\lambda}{2}g(V_1,V_2)|_x\notag\\
=&{\rm Ric}_{\phi^{\ast}g}(V_1,V_2)|_x-\frac{\lambda}{2}g_{\phi^{\ast}g}(V_1,V_2)|_x\notag\\
=&{\rm Ric}(\phi_{\ast}(V_1),\phi_{\ast}(V_2))|_y-\frac{\lambda}{2}g(\phi_{\ast}(V_1),\phi_{\ast}(V_2))|_y\notag\\
=&-\nabla_{\phi_{\ast}(V_1)}\nabla_{\phi_{\ast}(V_2)}f|_y\notag\\
=&-\bar{\nabla}_{V_1}\bar{\nabla}_{V_2}(f\circ\phi)|_x\notag\\
=&-\nabla_{V_1}\nabla_{V_2}(f\circ\phi)|_x\label{Parallel-vector-field-1}
\end{align}
where $\bar{\nabla}$ is the gradient with respect to the metric $\phi^{\ast}g$. By $(\ref{Parallel-vector-field-1})$, we have 
\begin{align}
\nabla(\nabla f-\nabla (f\circ \phi))\equiv0.\notag
\end{align}
Let $V=\nabla f-\nabla (f\circ \phi)$. Then, $V$ is a parallel vector field by the equation above. If $V$ is nonzero, then $(M,g)$ locally splits off a line. If $V\equiv0$, then there is a constant $c$ such that 
\begin{align}
f(x)-f(\phi(x))=c,~\forall~x\in M.\label{Parallel-vector-field-2}
\end{align}
Since $\phi\in G$ and $G$ is a finite group, then there is a positive integer $k$ such that $\phi^{k}={\rm Id}$. Hence, for any fixed point $x\in M$,
\begin{align}
kc=\sum_{i=0}^{k-1}[f(\phi^{i}(x))-f(\phi^{i+1}(x))]=f(x)-f(\phi^{k}(x))=0.\notag
\end{align}
It follows that 
\begin{align}
    c=0.\label{Parallel-vector-field-3}
\end{align}
By (\ref{Parallel-vector-field-2}) and (\ref{Parallel-vector-field-3}), we get 
\begin{align}
    f(x)=f(\phi(x)),~\forall~x\in M.
\end{align}
We complete the proof.

     \end{proof}

\begin{remark}
In Theorem \ref{theorem-symmetry of f}, we assume that $G$ is a finite subgroup of ${\rm Iso}(M,g)$. If the isomorphism $\phi$ is related to a killing vector field, one may refer to \cite{petersen-gradient-2009}.
\end{remark}

     As an application of Theorem \ref{theorem-symmetry of f} and Corollary \ref{corollary-Ricci positive case is simply connected}, we show that steady gradient Ricci solitons with nonnegative sectional curvature and positive Ricci curvature have no finite quotients.

     \begin{corollary}
         \label{corollary-sec>=0+Ric>0+finite =>simplyconnected}
         Let $(M,g,f)$ be a complete steady gradient Ricci soliton with nonnegative sectional curvature and positive Ricci curvature. Suppose $(N,h)$ is a finite quotient of $(M,g)$. Then, $N$ is simply connected, i.e., $(N,h)$ is isometric to $(M,g)$.
     \end{corollary}

     \begin{proof}
         By Corollary \ref{corollary-Ricci positive case is simply connected}, $M$ is simply connected. Hence, $M$ is a universal cover of $N$. Let $\pi: M\rightarrow N$ be the covering map.  By Proposition \ref{propositon-quotients isometric to universal cover module deck group}, $N$ is isometric to $M/{\rm Aut}_{\pi}(M)$. Note that ${\rm Aut}_{\pi}(M)$ is a finite group by assumption. Recall that  $(M,g,f)$ can not locally split off a line by Theorem \ref{split theorem_universal}. Therefore, for any $\phi\in {\rm Aut}_{\pi}(M)$, we have $f=f\circ\phi$ on $M$ by Theorem \ref{theorem-symmetry of f}. Hence, we can define a smooth function $\bar{f}$ on $N$ such that $f=\pi\circ \bar{f}$. Since $\pi$ is a local isomorphism, one can check that $(N,h,\bar{f})$ is a steady gradient Ricci soliton with nonnegative sectional curvature and positive Ricci curvature. By Corollary  \ref{corollary-Ricci positive case is simply connected}, $N$ is simply connected. Hence, we complete the proof.
     \end{proof}
	
	\section{The Proof of Theorem \ref{MAIN}}

Now we prove the main theorem. 
	\begin{proof}[Proof of Theorem \ref{MAIN}]
    It suffices to consider the case that $\pi_1(M)$ is finite. 
	Let $(\tilde{M} , \tilde{g},\tilde{f})$ be the universal cover of $(M,g,f)$ and $\pi : (\tilde{M} , \tilde{g}) \rightarrow (M,g)$ be the covering map. 
	By Theorem \ref{split theorem_universal},  
		either $(M,g,f)$ has strictly positive Ricci curvature
		or the universal cover $(\tilde{M}, \tilde{g})  = (N ,h ) \times \mathbb{R}^k$ 
        splits isometrically, where $(N , h)$ has strictly positive Ricci curvature and $k \ge 1$. Due to Lemma 2.1 in \cite{petersen-gradient-2009}, 
        $f$ can also split so that $N$ and $\mathbb{R}^k$ are both steady gradient Ricci solitons.
	If $M$ has strictly positive Ricci curvature, 
            then  $\pi_1(M)$ is trivial due to 
            Corollary \ref{corollary-Ricci positive case is simply connected}. 
    Now. it remains to consider the case $k\ge 1$.

    By Proposition \ref{proposition-isometry separately N R^n}, 
    \begin{align}
        {\rm Aut}_{\pi} (\tilde{M})\subseteq {\rm Iso} (N \times \mathbb{R}^k) = {\rm Iso} ( N ) \times {\rm Iso} ( \mathbb{R}^k)\notag
    \end{align}
   For any $\phi \in {\rm Aut}_{\pi} (\tilde{M})$, we can assume that  $\phi = (\phi_1 ,\phi_2)$, where $\phi_1 \in {\rm Iso}(N)$ and $\phi_2 \in {\rm Iso}(\mathbb{R}^k)$. 
     By Theorem \ref{theorem-structure theorem on universal covering map}, the fundamental group $\pi_1(M) \cong {\rm Aut}_{\pi} (\tilde{M})$. Since $\pi_1(M)$ is finite, we see that ${\rm Aut}_{\pi} (\tilde{M})$ is also finite. It follows that $\phi$, $\phi_1$ and $\phi_2$ all have finite orders.

     Let $G_{\phi}=\{\phi^i|i\in\mathbb{N}\}$. Then, $G_{\phi}$ is a finite subgroup of ${\rm Aut}_{\pi} (\tilde{M})$. By  Proposition 21.5 (c) in \cite{lee_introduction_2012} and Proposition \ref{Aut:smooth+free+proper+isometrical proposition}, $G_{\phi}$ acts smoothly, freely, properly and isometrically on $\tilde{M}$. Similarly, let $G_{\phi_2}=\{\phi_2^i|i\in\mathbb{N}\}$. Since $\phi_2$ has a finite order, $G_{\phi_2}$ is a finite subgroup of ${\rm Iso}(\mathbb{R}^k)$. So, $G_{\phi_2}$ is also compact. By Lemma 3.1.2 in \cite{wolf_spaces_2011}, 
        $G_{\phi_2}$  has at least one fixed point. Therefore, $\phi_2^i$ has at least one fixed point for all $i\in\mathbb{N}$.  We also define $G_{\phi_1}=\{\phi_1^i|i\in\mathbb{N}\}$. Then, $G_{\phi_1}$ is a finite subgroup of ${\rm Iso}(N)$.

    Let the order of $\phi$ be $r$. We want to show that $r=1$. Suppose $r\ge 2$. For $0< i<r$, $\phi^i$ has no fixed point on $N\times \mathbb{R}^k$ since $G_{\phi}$ is a free action. Note that $\phi^i=(\phi^i_1,\phi_2^i)$ and $\phi_2^i$ has at least a fixed point on $\mathbb{R}^k$. So, $\phi^i_1$ has no fixed point on $N$ for all $0<i<r$. Let the order of $\phi_1$ be $r_1$. Then, $r_1|r$ and therefore $r_1\le r$. Note that $\phi_1^i$ has no fixed point and cannot be the identity map for $0<i<r$. Hence, $r=r_1$. It follows that  $\phi^i_1$ has no fixed point on $N$ for all $0<i<r_1$.
    Then, $G_{\phi_1}$ is a free action on $N$ by definition. By Proposition 21.5 (c) in \cite{lee_introduction_2012}, $G_{\phi_1}$ acts freely and properly on $N$. By Theorem 21.10 in \cite{lee_introduction_2012}, $N/G_{\phi_1}$ is a quotient of $N$. However, $G_{\phi_1}$ is trivial by Corollary \ref{corollary-sec>=0+Ric>0+finite =>simplyconnected}. It contradicts the fact that $\phi_1^i$ has no fixed point for $0<i<r$. Hence, $r=1$ and $\phi$ is the identity map. Since $\phi$ is an arbitrary isomorphism in ${\rm Aut}_{\pi} (\tilde{M})$, we obtain that $\pi_1(M)\cong {\rm Aut}_{\pi} (\tilde{M})=\{{\rm Id}_{\tilde{M}}\}$.

Now we are left to show that $M$ is diffeomorphic to $\mathbb{R}^n$. We note that $(N,h)$ is a steady gradient Ricci soliton with nonnegative sectional curvature and positive Ricci curvature. Hence, $N$ is diffeomorphic to $\mathbb{R}^{n-k}$.  Then, $\tilde{M}$ is diffeomorphic to $\mathbb{R}^n$.  Since $\pi_1(M)$ is trivial, $M$ is diffeomorphic to $\tilde{M}$ and therefore is also diffeomorphic to $\mathbb{R}^n$.


    Hence we complete the proof.
	\end{proof}

\begin{spacing}{1.0}

{
    \footnotesize
    \bibliographystyle{plain}
    \bibliography{simplyconnected_Reference.bib}

\begin{thebibliography}{10}

\bibitem{appleton_family_2022}
Alexander Appleton.
\newblock A family of non-collapsed steady ricci solitons in even dimensions greater or equal to four.
\newblock arXiv:math.DG/1708.00161, 2022.

\bibitem{bamler_fundamental_2021}
Richard~H. Bamler.
\newblock On the fundamental group of non-collapsed ancient ricci flows.
\newblock arXiv:math.DG/2110.02254, 2021.

\bibitem{Bamler-Chow-Deng-Ma-2022}
Richard~H. Bamler, Bennett Chow, Yuxing Deng, Zilu Ma, and Yongjia Zhang.
\newblock Four-dimensional steady gradient {R}icci solitons with 3-cylindrical tangent flows at infinity.
\newblock {\em Adv. Math.}, 401:Paper No. 108285, 21, 2022.

\bibitem{brendle_rotational_2013}
Simon Brendle.
\newblock Rotational symmetry of self-similar solutions to the {R}icci flow.
\newblock {\em Invent. Math.}, 194(3):731--764, 2013.

\bibitem{brendle_rotational_2014}
Simon Brendle.
\newblock Rotational symmetry of {R}icci solitons in higher dimensions.
\newblock {\em J. Differential Geom.}, 97(2):191--214, 2014.

\bibitem{buttsworth-2021-arxiv}
Timothy Buttsworth.
\newblock $su(2)$-invariant steady gradient ricci solitons on four-manifolds.
\newblock arXiv:math.DG/2111.12807, 2021.

\bibitem{Buzano-Dancer-Wang-2015}
M.~Buzano, A.~S. Dancer, and M.~Wang.
\newblock A family of steady {R}icci solitons and {R}icci flat metrics.
\newblock {\em Comm. Anal. Geom.}, 23(3):611--638, 2015.

\bibitem{Cao-Catino-Chen-Mantegazza-Mazzieri-2014}
Huai-Dong Cao, Giovanni Catino, Qiang Chen, Carlo Mantegazza, and Lorenzo Mazzieri.
\newblock Bach-flat gradient steady {R}icci solitons.
\newblock {\em Calc. Var. Partial Differential Equations}, 49(1-2):125--138, 2014.

\bibitem{Cao_chen_2012}
Huai-Dong Cao and Qiang Chen.
\newblock On locally conformally flat gradient steady {R}icci solitons.
\newblock {\em Trans. Amer. Math. Soc.}, 364(5):2377--2391, 2012.

\bibitem{cao-xie-2024-arxiv}
Huai-Dong Cao and Junming Xie.
\newblock Four-dimensional gradient ricci solitons with (half) nonnegative isotropic curvature.
\newblock arXiv:math.DG/2403.19627, 2024.

\bibitem{chan_dimension_2023}
Pak-Yeung Chan, Zilu Ma, and Yongjia Zhang.
\newblock Dimension reduction for positively curved steady solitons.
\newblock arXiv:math.DG/2310.14020, 2023.

\bibitem{Cheeger-Gromoll-1972}
Jeff Cheeger and Detlef Gromoll.
\newblock On the structure of complete manifolds of nonnegative curvature.
\newblock {\em Ann. of Math. (2)}, 96:413--443, 1972.

\bibitem{Chen-Deruelle-2015}
Chih-Wei Chen and Alix Deruelle.
\newblock Structure at infinity of expanding gradient {R}icci soliton.
\newblock {\em Asian J. Math.}, 19(5):933--950, 2015.

\bibitem{chow_ricci_2023}
Bennett Chow.
\newblock {\em Ricci solitons in low dimensions}, volume 235 of {\em Graduate Studies in Mathematics}.
\newblock American Mathematical Society, Providence, RI, [2023] \copyright 2023.

\bibitem{Chow-Deng-Ma-2022}
Bennett Chow, Yuxing Deng, and Zilu Ma.
\newblock On four-dimensional steady gradient {R}icci solitons that dimension reduce.
\newblock {\em Adv. Math.}, 403:Paper No. 108367, 61, 2022.

\bibitem{conlon-Deruelle-arxiv-2020}
Ronan~J. Conlon and Alix Deruelle.
\newblock Steady gradient k\"ahler-ricci solitons on crepant resolutions of calabi-yau cones.
\newblock arXiv:math.DG/2006.03100, 2020.

\bibitem{Deng-Zhu-2018}
Yuxing Deng and Xiaohua Zhu.
\newblock Asymptotic behavior of positively curved steady {R}icci solitons.
\newblock {\em Trans. Amer. Math. Soc.}, 370(4):2855--2877, 2018.

\bibitem{Deng-Zhu-SCM-2020}
Yuxing Deng and Xiaohua Zhu.
\newblock Classification of gradient steady {R}icci solitons with linear curvature decay.
\newblock {\em Sci. China Math.}, 63(1):135--154, 2020.

\bibitem{Deng-Zhu-JEMS-2020}
Yuxing Deng and Xiaohua Zhu.
\newblock Higher dimensional steady {R}icci solitons with linear curvature decay.
\newblock {\em J. Eur. Math. Soc. (JEMS)}, 22(12):4097--4120, 2020.

\bibitem{Deng-Zhu-Mathann-2020}
Yuxing Deng and Xiaohua Zhu.
\newblock Rigidity of {$\kappa$}-noncollapsed steady {K}\"ahler-{R}icci solitons.
\newblock {\em Math. Ann.}, 377(1-2):847--861, 2020.

\bibitem{gromoll_complete_1969}
Detlef Gromoll and Wolfgang Meyer.
\newblock On complete open manifolds of positive curvature.
\newblock {\em Ann. of Math. (2)}, 90:75--90, 1969.

\bibitem{guan_rigidity_nodate}
Pengfei Guan, Peng Lu, and Yiyan Xu.
\newblock A rigidity theorem for codimension one shrinking gradient {R}icci solitons in {$\Bbb{R}^{n+1}$}.
\newblock {\em Calc. Var. Partial Differential Equations}, 54(4):4019--4036, 2015.

\bibitem{hamilton_formation_1995}
Richard~S. Hamilton.
\newblock The formation of singularities in the {R}icci flow.
\newblock In {\em Surveys in differential geometry, {V}ol.\ {II} ({C}ambridge, {MA}, 1993)}, pages 7--136. Int. Press, Cambridge, MA, 1995.

\bibitem{kobayashi_foundations_1996}
Shoshichi Kobayashi and Katsumi Nomizu.
\newblock {\em Foundations of differential geometry. {V}ol. {I}}.
\newblock Wiley Classics Library. John Wiley \& Sons, Inc., New York, 1996.
\newblock Reprint of the 1963 original, A Wiley-Interscience Publication.

\bibitem{Lai-JDG-2024}
Yi~Lai.
\newblock A family of 3{D} steady gradient solitons that are flying wings.
\newblock {\em J. Differential Geom.}, 126(1):297--328, 2024.

\bibitem{lee_introduction_2011}
John~M. Lee.
\newblock {\em Introduction to topological manifolds}, volume 202 of {\em Graduate Texts in Mathematics}.
\newblock Springer, New York, second edition, 2011.

\bibitem{lee_introduction_2012}
John~M. Lee.
\newblock {\em Introduction to smooth manifolds}, volume 218 of {\em Graduate Texts in Mathematics}.
\newblock Springer, New York, second edition, 2013.

\bibitem{lee_introduction_2018}
John~M. Lee.
\newblock {\em Introduction to {R}iemannian manifolds}, volume 176 of {\em Graduate Texts in Mathematics}.
\newblock Springer, Cham, second edition, 2018.

\bibitem{Liu-2013}
Gang Liu.
\newblock 3-manifolds with nonnegative {R}icci curvature.
\newblock {\em Invent. Math.}, 193(2):367--375, 2013.

\bibitem{ma-Mahmoudian-Sesum-arxiv-2023}
Zilu Ma, Hamidreza Mahmoudian, and Natasa Sesum.
\newblock Unique asymptotics of steady ricci solitons with symmetry.
\newblock arXiv:math.DG/2311.09405, 2023.

\bibitem{Munteanu_wang_2011}
Ovidiu Munteanu and Jiaping Wang.
\newblock Smooth metric measure spaces with non-negative curvature.
\newblock {\em Comm. Anal. Geom.}, 19(3):451--486, 2011.

\bibitem{perelman_entropy_2002}
Grisha Perelman.
\newblock The entropy formula for the ricci flow and its geometric applications.
\newblock arXiv:math.DG/0211159, 2002.

\bibitem{petersen_riemannian_2006}
Peter Petersen.
\newblock {\em Riemannian geometry}, volume 171 of {\em Graduate Texts in Mathematics}.
\newblock Springer, New York, second edition, 2006.

\bibitem{petersen-gradient-2009}
Peter Petersen and William Wylie.
\newblock On gradient {R}icci solitons with symmetry.
\newblock {\em Proc. Amer. Math. Soc.}, 137(6):2085--2092, 2009.

\bibitem{Schoen-Yau-1982}
Richard Schoen and Shing~Tung Yau.
\newblock Complete three-dimensional manifolds with positive {R}icci curvature and scalar curvature.
\newblock In {\em Seminar on {D}ifferential {G}eometry}, volume No. 102 of {\em Ann. of Math. Stud.}, pages 209--228. Princeton Univ. Press, Princeton, NJ, 1982.

\bibitem{sarafutdinov-1979}
V.~A. \v~Sarafutdinov.
\newblock Convex sets in a manifold of nonnegative curvature.
\newblock {\em Mat. Zametki}, 26(1):129--136, 159, 1979.

\bibitem{wink-2023}
Matthias Wink.
\newblock Cohomogeneity one {R}icci solitons from {H}opf fibrations.
\newblock {\em Comm. Anal. Geom.}, 31(3):625--676, 2023.

\bibitem{wolf_spaces_2011}
Joseph~A. Wolf.
\newblock {\em Spaces of constant curvature}.
\newblock AMS Chelsea Publishing, Providence, RI, sixth edition, 2011.

\bibitem{Wylie_2008}
William Wylie.
\newblock Complete shrinking {R}icci solitons have finite fundamental group.
\newblock {\em Proc. Amer. Math. Soc.}, 136(5):1803--1806, 2008.

\bibitem{zhao-zhu-arxiv-2022}
Ziyi Zhao and Xiaohua Zhu.
\newblock Rigidity of the bryant ricci soliton.
\newblock arXiv:math.DG/2212.02889, 2023.

\bibitem{zhao-zhu-arxiv-2024}
Ziyi Zhao and Xiaohua Zhu.
\newblock $4d$ steady gradient ricci solitons with nonnegative curvature away from a compact set.
\newblock arXiv:math.DG/2310.12529, 2024.

\end{thebibliography}
	}
    
\end{spacing}

	{\footnotesize YUXING DENG, SCHOOL OF MATHEMATICS AND STATISTICS, BEIJING INSTITUTE OF TECHNOLOGY, BEIJING,100081, CHINA, 6120180026@BIT.EDU.CN}
	
	{\footnotesize YUEHAN HAO, SCHOOL OF MATHEMATICS AND STATISTICS, BEIJING INSTITUTE OF TECHNOLOGY, BEIJING,100081, CHINA, 3120221441@BIT.EDU.CN}

\end{document}